\documentclass[12pt,reqno]{amsart}
\usepackage[numbers,sort&compress]{natbib}
\usepackage{amsfonts,bbm}
\usepackage{amssymb,color}
\usepackage{fancyhdr}
\usepackage[titletoc]{appendix}
\usepackage{enumitem}
\usepackage{amsgen}
\usepackage{amscd}
\usepackage{amsmath}
\usepackage{amsthm}
\usepackage{mathrsfs}
\usepackage{cases}
\usepackage{setspace}
\usepackage{accents}
\usepackage{verbatim}
\usepackage[colorlinks=true]{hyperref}
\hypersetup{urlcolor=blue,citecolor=green,linkcolor=blue,hypertexnames=false}
\usepackage[margin=3cm,a4paper]{geometry}

\newtheorem{thm}{Theorem}[section]

\newtheorem{lem}[thm]{Lemma}
\newtheorem{prop}[thm]{Proposition}
\newtheorem{defn}[thm]{\bfseries Definition}

\theoremstyle{remark}
\newtheorem{remark}[thm]{Remark}

\def\wt#1{\widetilde{#1}}

\newcommand{\R}{{\mathbb R}}
\newcommand{\C}{{\mathbb C}}

\newcommand{\Sch}{{\mathcal S}}

\def\d{\mathrm{\ d}}
\def\dx{\mathrm{\ d}x}

\newcommand{\re}{{\mathrm{Re}}}
\newcommand{\im}{{\mathrm{Im}}}

\numberwithin{equation}{section}

\begin{document}
	
	\title[ORBITAL STABILITY IN THE ENDPOINT CASE]{ORBITAL STABILITY OF SOLITARY WAVE SOLUTIONS OF THE DERIVATIVE NONLINEAR SCHR\"ODINGER EQUATION IN THE ENDPOINT CASE}
	\date{}
	
	\begin{abstract}\noindent
		In this paper, we consider the orbital stability of solitary waves for the derivative nonlinear Schr\"odinger equation
		\begin{equation*}
			iq_t+q_{xx}+i\partial_x\bigl(|q|^2q\bigr)=0,
			\quad (t,x)\in\R\times\R.
		\end{equation*}
		For the endpoint solitary waves satisfying $c>0$ and $c^2=4\omega$, we prove the orbital stability in $H^1(\R)$ without any additional assumptions. The proof uses conservation of mass and the second Hamiltonian $\mathcal{H}_2$ to keep the B\"acklund-transformed solution small in $H^1(\R)$ uniformly in time. We establish an inverse B\"acklund compactness result to reconstruct the algebraic solitary wave as the spectral zero approaches the negative real axis.
	\end{abstract}
	
	\author{Jia Shen}
	\address{(J. Shen) School of Mathematical Sciences and LPMC\\
		Nankai University\\
		Tianjin 300071, China}
	\email{shenjia@nankai.edu.cn}
	
	\author{Yifei Wu}
	\address{(Y. Wu) School of Mathematical Sciences\\
		Nanjing Normal  University\\
		Nanjing 210046, China}
	\email{yerfmath@gmail.com}
	
	\author{Changping Yang}
	\address{(C. Yang) Center for Applied Mathematics\\
		Tianjin University\\
		Tianjin 300072, China}
	\email{cp\_yang@tju.edu.cn}
	
	\maketitle
	
	\section{Introduction}
	\vspace{0.5cm}
	
	We consider the orbital stability of solitary wave solutions for the derivative nonlinear Schr\"odinger equation (DNLS)
	\begin{equation}\label{eq:dnls-equation}
		iq_t+q_{xx}+i\partial_x\bigl(|q|^2q\bigr)=0,
		\quad (t,x)\in\R\times\R.
	\end{equation}
	
	Before turning to solitary waves, we briefly recall the well-posedness theory for \eqref{eq:dnls-equation}. Hayashi and Ozawa \cite{HO-1992-PD} proved global existence in Sobolev spaces under an explicit smallness condition, and Takaoka \cite{Takaoka-1999-ADE} established local well-posedness in $H^{\frac{1}{2}}(\R)$. At the same regularity, Miao, Wu, and Xu \cite{MWX-2011-JDE} obtained global well-posedness under a mass restriction. Wu \cite{Wu-2013-APDE,Wu-2015-APDE} subsequently improved the mass threshold for global well-posedness in the energy space. The mass restriction was eventually removed by Bahouri and Perelman \cite{BP-2022-IM}, who proved global well-posedness for arbitrary initial data in $H^{\frac{1}{2}}(\R)$, and the result was later extended to the scaling-critical space $L^2(\R)$ by Harrop-Griffiths, Killip, Ntekoume, and Vi\c{s}an \cite{HGKNV-2026-JEMS}. A complementary global theory based on inverse scattering was developed in weighted Sobolev spaces by Jenkins, Liu, Perry, and Sulem. They first treated an open dense subset of $H^{2,2}(\R)$ containing initial data of arbitrarily large mass \cite{JLPS-2018-CPDE}, and then extended the result to arbitrary data in $H^{2,2}(\R)$ \cite{JLPS-2020-APDE}.
	
	A solitary wave with frequency $\omega>0$ and speed $c\in\R$ has the form
	\begin{equation}
		q_{\omega,c}(t,x)=e^{i\omega t}\phi_{\omega,c}(x-ct).
	\end{equation}
	Then $\phi_{\omega,c}$ satisfies
	\begin{equation}\label{eq:solitary-wave-profile-equation}
		-\phi_{\omega,c}''+\omega\phi_{\omega,c}+ic\phi_{\omega,c}'
		-i\partial_x\bigl(|\phi_{\omega,c}|^2\phi_{\omega,c}\bigr)=0.
	\end{equation}
	The solitary-wave profile is known explicitly; see~\cite{GW-1995-JDE,CO-2006-AIHP,KW-2018-JAM}. It can be written as
	\begin{equation*}
		\phi_{\omega,c}(x)=\varphi_{\omega,c}(x)
		\exp\left\{\frac{icx}{2}-\frac{3i}{4}\int_{-\infty}^{x}\varphi_{\omega,c}(y)^2\,dy\right\},
	\end{equation*}
	where $\varphi_{\omega,c}$ is the real-valued amplitude. When $c^2<4\omega$, we have
	\begin{equation}\label{eq:nonendpoint-amplitude-profile}
		\varphi_{\omega,c}(x)
		=\left[\frac{\sqrt\omega}{4\omega-c^2}
		\left\{\cosh\bigl(\sqrt{4\omega-c^2}\,x\bigr)-\frac{c}{2\sqrt\omega}\right\}\right]^{-1/2}.
	\end{equation}
	At the nontrivial endpoint $c=2\sqrt\omega>0$, \eqref{eq:nonendpoint-amplitude-profile} has the limit
	\begin{equation*}
		W_c(x):=\frac{2\sqrt c}{(1+c^2x^2)^{1/2}}.
	\end{equation*}
	The corresponding endpoint profile is
	\begin{equation}\label{eq:endpoint-complex-profile}
		\phi_{c^2/4,c}(x)
		=2i\sqrt c\,e^{icx/2}\frac{1-icx}{(1+icx)^2}.
	\end{equation}
	
	At the endpoint, the restriction $c>0$ is essential for \eqref{eq:dnls-equation}, rather than merely a choice of propagation direction. Indeed, the amplitude $\varphi=\varphi_{\omega,c}$ satisfies
	\begin{equation*}
		-\varphi''+\left(\omega-\frac{c^2}{4}\right)\varphi
		+\frac c2\varphi^3-\frac{3}{16}\varphi^5=0.
	\end{equation*}
	At $c^2=4\omega$, every localized solution satisfies
	\begin{equation*}
		(\varphi')^2=\frac1{16}\varphi^4(4c-\varphi^2).
	\end{equation*}
	Thus a nonzero localized endpoint profile exists only when $c>0$; for $c\leq0$, the only localized endpoint profile is the zero solution.
	
	\begin{defn}[Orbital stability]\label{defn:orbital-stability}
		Let $u(t,x)=e^{i\omega t}\Phi(x-ct)$ be a solitary-wave solution of \eqref{eq:dnls-equation}, with $\Phi\in H^1(\R)$, and define
		\begin{equation*}
			\mathcal O_\Phi:=\bigl\{e^{i\gamma}\Phi(\cdot-y):(\gamma,y)\in\R^2\bigr\}.
		\end{equation*}
		We say that $u(t,x)$ is orbitally stable in $H^1(\R)$ if, for every $\varepsilon>0$, there exists $\delta>0$ such that whenever $q_0\in H^1(\R)$ satisfies
		\begin{equation*}
			\operatorname{dist}_{H^1}(q_0,\mathcal O_\Phi)<\delta,
		\end{equation*}
		the corresponding solution $q$ exists globally and obeys
		\begin{equation*}
			\sup_{t\in\R}\operatorname{dist}_{H^1}\bigl(q(t),\mathcal O_\Phi\bigr)<\varepsilon.
		\end{equation*}
		Otherwise, $u(t,x)$ is said to be orbitally unstable.
	\end{defn}
	
	When $c^2<4\omega$, which we call the non-endpoint regime, the profile and its derivatives decay exponentially as $|x|\to\infty$. Writing $M(f):=\int_\R|f(x)|^2\,dx=\|f\|_{L^2}^2$, we have
	\begin{equation*}
		M(\phi_{\omega,c})=\int_{\R}|\phi_{\omega,c}|^2 \dx
		=8\tan^{-1}\left(\sqrt{\frac{2\sqrt\omega+c}{2\sqrt\omega-c}}\right)\in(0,4\pi).
	\end{equation*}
	Guo and Wu \cite{GW-1995-JDE} proved orbital stability for part of this parameter range, and Colin and Ohta \cite{CO-2006-AIHP} proved it throughout the range $c^2<4\omega$. After a standard gauge transformation, their variational argument uses the Lyapunov functional
	\begin{equation*}
		S_{\omega,c}(v):=E(v)+\frac{\omega}{2}M(v)+cP(v),
	\end{equation*}
	where $E$, $M$, and $P$ are the conserved energy, mass, and momentum, respectively. The gauged solitary-wave profile is a constrained minimizer of $S_{\omega,c}$, and coercivity of its second variation away from the phase and translation directions yields orbital stability.
	
	At the endpoint, however, this variational mechanism loses fixed-scale coercivity. Indeed, $W_c$ decays only as $|x|^{-1}$ and satisfies
	\begin{equation*}
		M(W_c)=\int_\R |W_c(x)|^2\,dx=4\pi,
	\end{equation*}
	for every $c>0$. The $L^2$-critical scaling $W_c(x)=c^{1/2}W_1(cx)$ preserves the mass, while the coefficient $\omega-c^2/4$ in the quadratic part of $S_{\omega,c}$ vanishes. Consequently, the conserved variational quantities do not separate different endpoint scales.
	
	This degeneration leads us to consider the integrable structure of DNLS. The DNLS equation admits the Kaup--Newell Lax pair \cite{KN-1978-JMP}, whose compatibility condition is \eqref{eq:dnls-equation}. The spectral data of its spatial part, in particular the reciprocal transmission coefficient and its discrete zeros, together with the associated B\"acklund transformations, provide the main tools of our proof.
	
	B\"acklund transformations have been used to prove stability results for several completely integrable equations; see, for example, \cite{MP-2012-IMRN,CPS-2016-CPDE,HW-2013-DIE,KT-2024-PMIHES}. For DNLS, Pelinovsky, Saalmann, and Shimabukuro \cite{PSS-2017-DPDE} used the Kaup--Newell B\"acklund transformation to remove interior zeros of the scattering coefficient, while Yang, Fan, and Liu \cite{YFL-2025-MZ} used the corresponding removal--reconstruction scheme to prove $L^2$ orbital stability of non-endpoint solitons. These arguments require an interior spectral point and do not cover the algebraic endpoint.
	
	For the algebraic endpoint, Klaus, Pelinovsky, and Rothos \cite{KPR-2006-JNS} analyzed the embedded spectral point, while Harrop-Griffiths, Killip, and Vi\c{s}an \cite{HGKV-2023-IMRN} showed that the transmission coefficient of the algebraic profile is identically one and hence does not determine its scale. Kwon and Wu \cite{KW-2018-JAM} proved conditional stability modulo scaling, phase, and translation, while Kim \cite{Kim-2024-JMAA} removed the non-smallness condition when the initial energy and momentum are nonpositive.
	
	The following is our main result.
	
	\begin{thm}\label{thm:main}
		Let $c>0$ and $c^2=4\omega$. Then $q_{\omega,c}$ is orbitally stable in $H^1(\R)$. 
	\end{thm}
	
	\begin{remark}
		Theorem~\ref{thm:main} extends the $H^1(\R)$ orbital stability result of Colin and Ohta \cite{CO-2006-AIHP} from the exponentially decaying solitary waves with $c^2<4\omega$ to the algebraically decaying solitary waves with $c^2=4\omega$ and $c>0$. The result requires neither the non-smallness condition on the $L^6(\R)$ norm of the solution in \cite{KW-2018-JAM} nor the sign restrictions on the initial energy and momentum in \cite{Kim-2024-JMAA}. Every initial datum sufficiently close to the endpoint solitary-wave orbit in $H^1(\R)$ generates a global solution that remains close to this orbit for all time.
	\end{remark}
	
	\begin{remark}[Reflection]
		Let $\wt{q}(t,x):=q(t,-x)$. Then $\wt{q}$ solves
		\begin{equation}\label{eq:dnls-reflected}
			i\wt{q}_t+\wt{q}_{xx}-i\partial_x\bigl(|\wt{q}|^2\wt{q}\bigr)=0.
		\end{equation}
		Define
		\begin{equation*}
			\widetilde q_{\omega,c}(t,x):=q_{\omega,-c}(t,-x).
		\end{equation*}
		Then $\widetilde q_{\omega,c}$ is an endpoint solitary wave of speed $c$ for \eqref{eq:dnls-reflected}. Since spatial reflection is an isometry on $H^1(\R)$, Theorem~\ref{thm:main} yields orbital stability of $\widetilde q_{\omega,c}$ in $H^1(\R)$ for \eqref{eq:dnls-reflected} when $c^2=4\omega$ and $c<0$.
	\end{remark}
	
	\subsection{Outline of the proof of Theorem \ref{thm:main}}
	
	Our proof first establishes orbital stability for Schwartz initial data and then extends the result to $H^1(\R)$ by approximation. For Schwartz data, we use the integrable structure of DNLS to select a spectral zero near the algebraic endpoint. As the initial data approach the algebraic profile in $H^1(\R)$, their perturbation determinants converge uniformly to one on the positive imaginary axis. Together with the convergence of the mass to $4\pi$, the mass--zeros--winding identity yields exactly one zero in the second quadrant. We also show that this zero remains bounded and stays away from the origin, and that its possible limits lie on the negative real axis.
	
	We remove the selected zero by a B\"acklund transformation. The mass shift identity shows that the transformed function tends to zero in $L^2(\R)$. Since the second Hamiltonian vanishes on the algebraic profile, the Hamiltonian shift identity also shows that the second Hamiltonian of the transformed function tends to zero. For functions with sufficiently small mass, this Hamiltonian controls the $L^2(\R)$ norm of the spatial derivative. We therefore obtain convergence of the transformed function to zero in $H^1(\R)$.
	
	Next, we study the inverse B\"acklund transformation to reconstruct the original function. The removal--reconstruction arguments for non-endpoint solitons work at an interior spectral point. In the endpoint limit considered here, the selected zero approaches the negative real axis, while the determinant of the limiting algebraic profile is identically one. We prove an inverse compactness result for this boundary limit: if the transformed functions tend to zero in $H^1(\R)$ and the selected zeros approach a nonzero point on the negative real axis, then the original functions approach the corresponding algebraic-soliton orbit in $H^1(\R)$. We first obtain local convergence of the reconstruction data from their first-order equations, and then use the mass identity and these equations to obtain global $H^1$ convergence. Applying this result to the initial data and using the positive separation of orbits at different scales identifies the limiting spectral point with the scale of the original profile. This step allows us to retain the fixed scale rather than introduce scaling modulation.
	
	We then apply these results along the DNLS flow. The perturbation determinant is preserved, so the selected zero does not change with time. By Darboux covariance, the transformed function also solves DNLS. Conservation of its mass and second Hamiltonian, together with the small-mass estimate above, keeps it small in $H^1(\R)$ uniformly in time. Applying the inverse compactness result to sequences of arbitrary times then proves orbital stability for Schwartz initial data.
	
	Finally, we pass to general $H^1(\R)$ initial data by Schwartz approximation and continuous dependence on finite time intervals. If stability failed for a sequence of $H^1(\R)$ initial data approaching the solitary-wave orbit, we could approximate each datum by a Schwartz function closely enough that the corresponding solutions remain close up to the chosen time of departure from the orbit. The approximating data still approach the same orbit, so the Schwartz stability result gives a contradiction. This proves the $H^1(\R)$ orbital stability stated in Theorem \ref{thm:main}.
	
	\subsection{Organization of the paper}
	The remainder of the paper is organized as follows. Section~\ref{sec:preliminaries} is devoted to the spectral theory, the B\"acklund transformation, and well-posedness, which will be used later. Section~\ref{sec:spectral-analysis} establishes the spectral-zero selection and the endpoint inverse B\"acklund compactness. Section~\ref{sec:proof-main-theorem} proves Theorem~\ref{thm:main}.
	
	\vspace{1cm}
	\section{Preliminaries}\label{sec:preliminaries}
	\vspace{0.5cm}

	In the following, we record some notation and useful lemmas. Throughout, $\operatorname{Log}z=\log|z|+i\arg z$ denotes the principal logarithm on $\C\setminus(-\infty,0]$, with $\arg z\in(-\pi,\pi)$. We denote the real part of $z$ by $\re \ z$ and the imaginary part by $\im \ z$.
	
	\subsection{The perturbation determinant}
	
	Let
	\begin{equation*}
		\sigma_3:=\begin{pmatrix}1&0\\0&-1\end{pmatrix},
		\quad \mathcal P(q):=\begin{pmatrix}0&q\\-\bar q&0\end{pmatrix}.
	\end{equation*}
	Kaup and Newell \cite{KN-1978-JMP} associated the following Lax pair with \eqref{eq:dnls-equation}:
	\begin{equation}\label{eq:kaup-newell-lax-pair}
		\Psi_x=U(q,\lambda)\Psi,\quad \Psi_t=V(q,\lambda)\Psi,
	\end{equation}
	where
	\begin{equation*}
		\begin{aligned}
			U(q,\lambda)
			&=-i\lambda^2\sigma_3+\lambda\mathcal P(q),\\
			V(q,\lambda)
			&=-2i\lambda^4\sigma_3+2\lambda^3\mathcal P(q)
			+i\lambda^2|q|^2\sigma_3-\lambda|q|^2\mathcal P(q)
			+i\lambda\sigma_3\mathcal P(q_x).
		\end{aligned}
	\end{equation*}
	Its compatibility condition is
	\begin{equation*}
		U_t-V_x+[U,V]=0
		\quad\Longleftrightarrow\quad
		iq_t+q_{xx}+i\partial_x(|q|^2q)=0.
	\end{equation*}
	Thus DNLS is the isospectral evolution of the spatial Kaup--Newell problem, and the coefficient $a$ defined below is preserved by the DNLS flow.
	
	For $f\in\Sch(\R)$ and $\lambda^2\in\R$, let $\Psi_-$ solve the spatial equation in \eqref{eq:kaup-newell-lax-pair} with
	\begin{equation*}
		\lim_{x\to-\infty}\Psi_-(x;\lambda)e^{i\lambda^2x\sigma_3}=I.
	\end{equation*}
	Then we define the scattering matrix
	\begin{equation*}
		\mathsf S(\lambda;f)
		:=\lim_{x\to+\infty}e^{i\lambda^2x\sigma_3}\Psi_-(x;\lambda)
		=
		\begin{pmatrix}
			a_0(\lambda;f)&-\overline{b_0(\bar\lambda;f)}\\
			b_0(\lambda;f)&\overline{a_0(\bar\lambda;f)}
		\end{pmatrix}.
	\end{equation*}
	We call $a_0(\lambda;f)$ the reciprocal transmission coefficient. Set $\C_{\mathrm I}:=\{\lambda\in\C:\re\lambda>0,\ \im\lambda>0\}$. Then $a_0(\lambda;f)$ admits a natural extension for $\lambda\in \C_{\mathrm I}$.
	
	For $k\in \C^+:=\{z\in \C: \im \ z>0\}$ and $f\in L^2(\R)$, we define
	\begin{equation}\label{eq:perturbation-determinant}
		a(k;f):=\det\bigl[1-k\Lambda(k;f)\Gamma(k;f)\bigr],
	\end{equation}
	where
	\begin{equation}\label{eq:resolvent-factor-operators}
		\begin{aligned}
			\Lambda(k;f)&:=(-ik-\partial_x)^{-1/2}M_f(-ik+\partial_x)^{-1/2},\\
			\Gamma(k;f)&:=(-ik+\partial_x)^{-1/2}M_{\bar f}(-ik-\partial_x)^{-1/2}.
		\end{aligned}
	\end{equation}
	Here $M_f$ denotes multiplication by $f$, and the fractional powers are defined by the principal holomorphic branches of the Fourier multipliers. We use $\sqrt{k}:=|k|^{1/2}e^{i\arg k/2}$ with $0<\arg k<\pi$. We write $\mathfrak I_p$ for the Schatten $p$-class; in particular, $\mathfrak I_2$ is the Hilbert--Schmidt class and $\mathfrak I_1$ is the trace class. Moreover, we have the following inequality
	\begin{equation}\label{inequality-schatten}
		\|T_1T_2\|_{\mathfrak I_p}\lesssim \|T_1\|_{\mathfrak I_{p_1}} \|T_2\|_{\mathfrak I_{p_2}}
	\end{equation}
	for $\frac{1}{p}=\frac{1}{p_1}+\frac{1}{p_2}$, when $p,\ p_1,\ p_2\geq 1$.
	
	By the Jost-Pais identity \cite{HGKV-2023-IMRN}, for $k\in \C^+$ and $q\in\mathcal{S}(\R)$, we have
	\begin{equation}
		a(k;q)=a_0(\sqrt{k};q).
	\end{equation}
	
	We collect the properties of $a(k;f)$ used below; see \cite{HGKV-2023-IMRN}.
	
	\begin{lem}[Basic properties of the perturbation determinant]\label{lem:perturbation-determinant}
		The following statements hold:
		\begin{enumerate}[label=\textup{(\roman*)}]
			\item Let $0\in I\subset\R$ be an interval. If $q\in C^1(I;\Sch(\R))$ is a solution of \eqref{eq:dnls-equation}, then
			\begin{equation*}
				a(k;q(t))=a(k;q(0)),\quad t\in I,\quad k\in\C^+.
			\end{equation*}
			
			\item For every $f\in L^2(\R)$ and $k\in\C^+$, $\Lambda(k;f),\Gamma(k;f)\in\mathfrak I_2$ and $k\Lambda(k;f)\Gamma(k;f)\in\mathfrak I_1$. Moreover,
			\begin{equation}\label{eq:hilbert-schmidt-ray-bounds}
				\sqrt{\im k}\,
				\bigl(
				\|\Lambda(k;f)\|_{\mathfrak I_2}
				+\|\Gamma(k;f)\|_{\mathfrak I_2}
				\bigr)
				\lesssim\|f\|_{L^2}.
			\end{equation}
			
			\item Let $f\in L^2(\R)$. For $z\in\C^+$, $a(z;f)=0$ if and only if the spatial equation in \eqref{eq:kaup-newell-lax-pair} at $\lambda=\sqrt z$ admits a nonzero solution
			\begin{equation*}
				\psi\in H^1(\R;\C^2).
			\end{equation*}
			Such a solution is unique up to multiplication by a nonzero constant and is nowhere vanishing.
			
			\item Let $\mathcal F$ be bounded in $L^2(\R)$ and satisfy
			\begin{equation*}
				\lim_{h\to0}\sup_{g\in\mathcal F}
				\|g(\cdot+h)-g\|_{L^2}=0.
			\end{equation*}
			Set $M(\mathcal F):=\sup_{g\in\mathcal F}M(g)$. For every sufficiently small $\varepsilon>0$, there exists $N=N(\varepsilon,\mathcal F)>0$ such that every zero $z\in\C^+$ of $a(\cdot;g)$, $g\in\mathcal F$, satisfies
			\begin{equation}\label{eq:zero-free-region-estimate}
				|z|\leq\frac{N}{\varepsilon^2}
				\quad\text{or}\quad
				\arg z\leq C_{M(\mathcal F)}\varepsilon.
			\end{equation}
			
			\item For every $f\in L^2(\R)$, along the positive imaginary axis,
			\begin{equation}\label{eq:determinant-ray-endpoint-limits}
				\lim_{\kappa\downarrow0}a(i\kappa;f)=1,
				\quad
				\lim_{\kappa\to\infty}a(i\kappa;f)=e^{-iM(f)/2},
				\quad M(f):=\|f\|_{L^2}^2.
			\end{equation}
			Moreover, for $k\in\C^+$,
			\begin{equation}\label{a=1-qa}
				a(k;q_a)\equiv1 \quad \text{with}\quad q_a(x)=2e^{ix/2}\frac{1-ix}{(1+ix)^2}.
			\end{equation}
		\end{enumerate}
	\end{lem}

	\begin{lem}\label{lem:determinant-continuity}
		If $(k_n,f_n)\to(k,f)$ in $\C^+\times L^2(\R)$ as $n\to\infty$, then
		\begin{equation}\label{eq:determinant-L2-continuity}
			a(k_n;f_n)\to a(k;f) \quad \text{as}\quad n\to \infty.
		\end{equation}
	\end{lem}
	
	\begin{proof}
		Write $\Lambda_n=\Lambda(k_n;f_n)$ and $\Gamma_n=\Gamma(k_n;f_n)$. On a compact set $K\Subset\C^+$ containing $k_n$ and $k$ for all large $n$, the multiplier definitions and dominated convergence yield
		\begin{equation*}
			\begin{aligned}
				\|\Lambda_n-\Lambda(k;f)\|_{\mathfrak I_2}
				&\leq C_K\|f_n-f\|_{L^2}+\|\Lambda(k_n;f)-\Lambda(k;f)\|_{\mathfrak I_2}\to0,\\
				\|\Gamma_n-\Gamma(k;f)\|_{\mathfrak I_2}
				&\leq C_K\|f_n-f\|_{L^2}+\|\Gamma(k_n;f)-\Gamma(k;f)\|_{\mathfrak I_2}\to0,
			\end{aligned}
		\end{equation*}
		as $n\to\infty$. Here the last terms converge by dominated convergence, while the first terms use \eqref{eq:hilbert-schmidt-ray-bounds}. Hence, with $A_n=k_n\Lambda_n\Gamma_n$ and $A=k\Lambda(k;f)\Gamma(k;f)$,
		\begin{equation*}
			\begin{aligned}
				\|A_n-A\|_{\mathfrak I_1}
				&\leq |k_n-k|\|\Lambda_n\|_{\mathfrak I_2}\|\Gamma_n\|_{\mathfrak I_2}
				+|k|\|\Lambda_n-\Lambda(k;f)\|_{\mathfrak I_2}\|\Gamma_n\|_{\mathfrak I_2}\\
				&\quad+|k|\|\Lambda(k;f)\|_{\mathfrak I_2}\|\Gamma_n-\Gamma(k;f)\|_{\mathfrak I_2}
				\to0 \quad\text{as }n\to\infty.
			\end{aligned}
		\end{equation*}
		The continuity inequality for Fredholm determinants yields
		\begin{equation*}
			|a(k_n;f_n)-a(k;f)|
			\leq \|A_n-A\|_{\mathfrak I_1}
			\exp\bigl(1+\|A_n\|_{\mathfrak I_1}+\|A\|_{\mathfrak I_1}\bigr)\to0
		\end{equation*}
		as $n\to\infty$, which proves \eqref{eq:determinant-L2-continuity}.
	\end{proof}
	
	By Lemma~\ref{lem:perturbation-determinant}(iii), every zero $z\in\C^+$ of $a(\cdot;f)$ provides the eigenfunction required to define the following B\"acklund transform.
	
	\subsection{The B\"acklund transformation}
	
	\begin{defn}[B\"acklund transform]\label{defn:backlund-transform}
		Let $f\in\Sch(\R)$, and let $z\in\C^+$ be a zero of $a(\cdot;f)$. Choose $\mu=\sqrt z$ in the first quadrant and a nonzero $L^2$-eigenfunction $\psi=(\psi_1,\psi_2)^T$ of the spatial part of \eqref{eq:kaup-newell-lax-pair} at $\lambda=\mu$. Define
		\begin{equation}\label{eq:backlund-transform-data}
			\begin{aligned}
				d&:=\mu|\psi_1|^2+\bar\mu|\psi_2|^2,
				& G&:=\frac{\bar d}{d},\\
				S&:=\frac{4\im\  z\,\psi_1\bar\psi_2}{d},
				& \mathbf B_z(f)&:=G^2f+GS.
			\end{aligned}
		\end{equation}
		Since $\psi$ is nowhere vanishing,
		\begin{equation*}
			\re \  d=(\re \ \mu)(|\psi_1|^2+|\psi_2|^2)>0.
		\end{equation*}
		Hence $\mathbf B_z(f)$ is well-defined and is independent of the normalization of $\psi$.
	\end{defn}
	
	The B\"acklund identities below record the zero removal, mass shift, and Hamiltonian shift; see \cite{HGKV-2023-IMRN,KN-1978-JMP}. We use
	\begin{equation}\label{eq:second-hamiltonian}
		\mathcal H_2(f)
		:=\int_\R\left[
		|f_x|^2+\frac{3i}{4}|f|^2(f\bar f_x-\bar f f_x)
		+\frac12|f|^6
		\right]\,dx.
	\end{equation}
	
	The identities below, except for \eqref{eq:backlund-hamiltonian-shift}, follow from \cite{HGKV-2023-IMRN}. We derive \eqref{eq:backlund-hamiltonian-shift} below from the trace expansion in \cite{BP-2022-IM}.
	
	\begin{lem}[B\"acklund identities]\label{lem:backlund-identities}
		Let $f,z,G,S$, and $\mathbf B_z(f)$ be as in Definition~\ref{defn:backlund-transform}. Write $z=|z|e^{i\varphi}$, $0<\varphi<\pi$. Then $\mathbf B_z(f)\in\Sch(\R)$,
		\begin{equation*}
			|G|=1,\quad
			G(-\infty)=e^{-i\varphi},\quad
			G(+\infty)=e^{i\varphi},
		\end{equation*}
		and
		\begin{align}
			|z|(G+\bar G)-\frac14|S|^2&=2\re z,
			&\quad S'&=\frac{2|z|}{i}\bigl(\mathbf B_z(f)-f\bigr),
			\label{eq:backlund-algebraic-differential-identities}\\
			|f|^2-|\mathbf B_z(f)|^2&=-2i\bar GG',
			&\quad \|S\|_\infty&\leq4|\im\sqrt z|\leq4\sqrt{|z|}.
			\label{eq:backlund-density-bound-identities}
		\end{align}
		Furthermore,
		\begin{align}
			\label{eq:backlund-mass-shift}
			M(\mathbf B_z(f))&=M(f)-4\arg z,\\
			\label{eq:backlund-determinant-shift}
			a(k;\mathbf B_z(f))
			&=\frac{z}{\bar z}\frac{k-\bar z}{k-z}a(k;f),\\
			\label{eq:backlund-hamiltonian-shift}
			\mathcal H_2(\mathbf B_z(f))
			&=\mathcal H_2(f)-8\im(z^2).
		\end{align}
	\end{lem}
	
	\begin{proof}
		We prove only \eqref{eq:backlund-hamiltonian-shift}.
		The remaining assertions follow from
		\cite{HGKV-2023-IMRN}.
		
		Set
		\begin{equation*}
			\mathcal{H}_1(f):=\im\int_\R \bar f f_x\,dx
			+\frac12\int_\R |f|^4\,dx.
		\end{equation*}
		By \cite[Section~2.1]{BP-2022-IM}, the following expansion holds for $k=i\kappa$ as $\kappa\to\infty$:
		\begin{equation*}
			\log a(k;f)
			=-\frac{i}{2}M(f)
			+\frac{i}{4k}\mathcal{H}_1(f)
			-\frac{i}{8k^2}\mathcal H_2(f)
			+O(\kappa^{-3}),
		\end{equation*}
		where the branch is chosen so that
		$\log a(k;f)\to-iM(f)/2$.
		
		Let $g=\mathbf B_z(f)$. By
		\eqref{eq:backlund-determinant-shift},
		\begin{equation*}
			\begin{aligned}
				\log a(k;g)-\log a(k;f)
				&=C+\log\left(1-\frac{\bar z}{k}\right)
				-\log\left(1-\frac{z}{k}\right)\\
				&=C+\frac{z-\bar z}{k}
				+\frac{z^2-\bar z^2}{2k^2}
				+O(\kappa^{-3}),
			\end{aligned}
		\end{equation*}
		where $C$ is independent of $k$ and $e^C=z/\bar z$.
		Comparing the $k^{-2}$ coefficients yields
		\begin{equation*}
			-\frac{i}{8}
			\bigl[\mathcal H_2(g)-\mathcal H_2(f)\bigr]
			=\frac{z^2-\bar z^2}{2}
			=i\im(z^2),
		\end{equation*}
		proving \eqref{eq:backlund-hamiltonian-shift}.
	\end{proof}

	The Darboux covariance of the B\"acklund transform is standard; see \cite{XHW-2011-JPA,PSS-2017-DPDE,YFL-2025-MZ}.
	
	\begin{lem}[Darboux covariance]\label{lem:darboux}
		Let $0\in I\subset\R$ be an interval, and let $q\in C^1(I;\Sch(\R))$ be a solution of \eqref{eq:dnls-equation}. Fix $z\in\C^+$ such that $a(z;q(0))=a(z;q(t))=0$ for every $t\in I$. Then
		\begin{equation}\label{eq:backlund-removed-field-definition}
			r(t):=\mathbf B_z(q(t))
		\end{equation}
		is a classical solution of \eqref{eq:dnls-equation} on $I$.
	\end{lem}

	\subsection{The mass--zeros--winding formula}
	
	The preceding transformation removes a prescribed zero. The number of zeros in a sector is determined by the following argument-principle identity; see \cite{BP-2022-IM,HGKV-2023-IMRN}.
	
	\begin{lem}[Mass--zeros--winding]\label{lem:mass-zeros-winding}
		Let $f\in\Sch(\R)$, $0<\theta\leq\pi$, and suppose that $a(\cdot;f)$ has no zero on $e^{i\theta}\R^+$. Then
		\begin{equation}\label{eq:mass-zeros-winding-formula}
			M(f)=4\pi N_\theta(f)
			-\frac2i\int_0^{\infty e^{i\theta}}
			\frac{\partial_ka(k;f)}{a(k;f)}\,dk,
		\end{equation}
		where $N_\theta(f)$ is the number of zeros of $a(\cdot;f)$, counted with algebraic multiplicity, in
		\begin{equation*}
			\{k\in\C^+:\theta<\arg k<\pi\}.
		\end{equation*}
	\end{lem}
	
	\subsection{Well-posedness}
	
	We will use the following global $H^1(\R)$ well-posedness theory; see \cite{BP-2022-IM,HGKNV-2026-JEMS}.
	
	\begin{prop}[Global well-posedness]\label{prop:h1-well-posedness}
		For every $q_0\in H^1(\R)$, the equation \eqref{eq:dnls-equation} admits a unique global solution
		\begin{equation*}
			q\in C(\R;H^1(\R)).
		\end{equation*}
		For every $T>0$, the solution map is continuous from $H^1(\R)$ to $C([-T,T];H^1(\R))$. Moreover, for $q_0\in\Sch(\R)$, $q\in C^1(\R;\Sch(\R))$.
	\end{prop}
	
	\vspace{1cm}
	\section{Spectral analysis}\label{sec:spectral-analysis}
	\vspace{0.5cm}
	
	For convenience, we set $\rho:=c/4>0$ and define
	\begin{equation*}
		Q_\rho(x):=4\sqrt\rho\,e^{2i\rho x}\frac{1-4i\rho x}{(1+4i\rho x)^2}.
	\end{equation*}
	By \eqref{eq:endpoint-complex-profile},
	\begin{equation}\label{eq:endpoint-profile-Q-correspondence}
		\phi_{c^2/4,c}(x)=iQ_{\rho}(x).
	\end{equation}
	Consequently,
	\begin{equation*}
		q_\rho(t,x):=e^{4i\rho^2t}Q_\rho(x-4\rho t)
	\end{equation*}
	is a right-traveling endpoint solitary wave of \eqref{eq:dnls-equation}. Define its orbit by
	\begin{equation*}
		\mathcal O_\rho:=\bigl\{e^{i\gamma}Q_\rho(\cdot-y):(\gamma,y)\in\R^2\bigr\}.
	\end{equation*}
	Hence $\mathcal O_{\phi_{c^2/4,c}}=\mathcal O_{\rho}$, where
	\begin{equation*}
		\mathcal O_{\phi_{c^2/4,c}} =\bigl\{e^{i\gamma}\phi_{c^2/4,c}(\cdot-y):(\gamma,y)\in\R^2\bigr\}.
	\end{equation*}
	Therefore, Theorem~\ref{thm:main} reduces to proving the orbital stability of $\mathcal O_\rho$.
	
	\subsection{Uniform convergence on the positive imaginary ray}

	\begin{lem}\label{lem:uniform-det}
		Suppose that $f,g\in L^2(\R)$ and $\|f\|_{L^2}+\|g\|_{L^2}\leq M_0$ for some $M_0>0$. Then there exists a constant $C_{M_0}>0$ such that
		\begin{equation}\label{eq:uniform-determinant-lipschitz-bound}
			\sup_{\kappa>0}|a(i\kappa;f)-a(i\kappa;g)|\leq C_{M_0}\|f-g\|_{L^2}.
		\end{equation}
		Furthermore,
		\begin{equation}\label{eq:algebraic-profile-trivial-determinant}
			a(k;Q_\rho)\equiv1\quad \text{for }k\in\C^+.
		\end{equation}
		Consequently, if $f_n\to Q_\rho\text{ in }L^2(\R)$ as $n\to\infty$, then we have
		\begin{equation*}
			\sup_{\kappa>0}|a(i\kappa;f_n)-1|\to 0\quad \text{as}\quad n\to \infty.
		\end{equation*}
	\end{lem}
	
	\begin{proof}
		Set
		\begin{equation*}
			A_f(\kappa):=i\kappa\Lambda(i\kappa;f)\Gamma(i\kappa;f).
		\end{equation*}
		Then
		\begin{equation*}
			a(i\kappa;f)=\det(1-A_f(\kappa)).
		\end{equation*}
		Since $\Lambda$ is linear in $f$ and $\Gamma$ is linear in $\bar f$, by \eqref{inequality-schatten} and \eqref{eq:hilbert-schmidt-ray-bounds}, we have
		\begin{equation}\label{eq:trace-class-difference-bound}
			\begin{aligned}
				\|A_f(\kappa)-A_g(\kappa)\|_{\mathfrak I_1}
				&\leq\kappa\|\Lambda(i\kappa;f)-\Lambda(i\kappa;g)\|_{\mathfrak I_2}\|\Gamma(i\kappa;f)\|_{\mathfrak I_2} \\
				&\qquad+\kappa\|\Lambda(i\kappa;g)\|_{\mathfrak I_2}\|\Gamma(i\kappa;f)-\Gamma(i\kappa;g)\|_{\mathfrak I_2}\\
				&\lesssim(\|f\|_{L^2}+\|g\|_{L^2})\|f-g\|_{L^2}.
			\end{aligned}
		\end{equation}
		Therefore, \eqref{eq:uniform-determinant-lipschitz-bound} follows from the following Fredholm determinant inequality
		\begin{equation*}
			|\det(1-A)-\det(1-B)|\leq\|A-B\|_{\mathfrak I_1}\exp\bigl(1+\|A\|_{\mathfrak I_1}+\|B\|_{\mathfrak I_1}\bigr).
		\end{equation*}
		
		Next, we verify \eqref{eq:algebraic-profile-trivial-determinant}. Note that
		\begin{equation*}
			Q_\rho(x)=\sqrt{4\rho}\,q_a(4\rho x),
		\end{equation*}
		where $q_a$ is defined in \eqref{a=1-qa}. Changing variables in \eqref{eq:resolvent-factor-operators}, we have
		\begin{equation*}
			a(k;\sqrt{h}\,q(h\,\cdot))=a(k/h;q).
		\end{equation*}
		Thus, we obtain \eqref{eq:algebraic-profile-trivial-determinant}.
	\end{proof}
	
	\subsection{Winding and the unique second-quadrant zero}
	
	\begin{lem}\label{lem:unique-zero}
		Let $f_n\in\Sch(\R)$ and suppose that $f_n\to Q_\rho$ in $H^1(\R)$. For all sufficiently large $n$, $a(\cdot;f_n)$ has exactly one zero, counted with algebraic multiplicity, in
		\begin{equation*}
			\Sigma_{\rm II}:=\{k\in\C^+:\pi/2<\arg k<\pi\}.
		\end{equation*}
	\end{lem}
	
	\begin{proof}
		By Lemma~\ref{lem:uniform-det}, for sufficiently large $n$, we have $ a(i\kappa;f_n)\in B(1,\frac{1}{2})$ for all $\kappa>0$. By Lemma \ref{lem:mass-zeros-winding}, 		
		\begin{equation*}
			M(f_n)=4\pi N_n
			-\frac2i\int_0^{i\infty}
			\frac{\partial_ka(k;f_n)}{a(k;f_n)}\d k,
		\end{equation*}
		where $N_n$ is the number of zeros of $a(\cdot; f_n)$ in $\Sigma_{\rm II}$. For $0<\varepsilon<R$, the fundamental theorem of calculus shows that
		\begin{equation*}
			\int_{i\varepsilon}^{iR}\frac{\partial_ka(k;f_n)}{a(k;f_n)}\,dk=\operatorname{Log}a(iR;f_n)-\operatorname{Log}a(i\varepsilon;f_n).
		\end{equation*}
		For sufficiently large $n$,
		\begin{equation*}
			\operatorname{Log}\left( e^{-i\frac{M(f_n)}{2}} \right)=i\left(2\pi-\frac{M(f_n)}{2}\right).
		\end{equation*}
		Letting $\varepsilon\to 0$ and $R\to\infty$ and using
		\eqref{eq:determinant-ray-endpoint-limits}, we obtain
		\begin{equation*}
			\int_0^{i\infty}\frac{\partial_ka(k;f_n)}{a(k;f_n)}\,dk=i\left(2\pi-\frac{M(f_n)}{2}\right).
		\end{equation*}
		Thus,
		\begin{equation*}
			M(f_n)=4\pi N_n-\frac2i\,i\left(2\pi-\frac{M(f_n)}{2}\right).
		\end{equation*}
		Therefore, we obtain $N_n=1$.
	\end{proof}
	
	\subsection{Spectral-zero selection}
	Next, we identify the limiting spectral zero as $-\rho$ and show that the corresponding B\"acklund transform tends to zero in $H^1(\R)$. More precisely,
	\begin{prop}\label{prop:scale}
		Let $f_n\in\Sch(\R)$ satisfy $f_n\to Q_\rho$ in $H^1(\R)$ as $n\to\infty$, let $z_n$ be the unique zero supplied by Lemma~\ref{lem:unique-zero}, and denote
		\begin{equation*}
			r_n:=\mathbf B_{z_n}(f_n).
		\end{equation*}
		Then
		\begin{equation}\label{eq:endpoint-scale-selection}
			z_n\to-\rho,\quad  r_n\to0\quad\text{in }H^1(\R) \text{ as } n\to\infty.
		\end{equation}
	\end{prop}
	
	The proposition follows once Propositions~\ref{prop:preclassification} and \ref{prop:endpoint-inverse} have been established below.
	
	\subsection{Classification of possible zero limits}
	
	\begin{prop}\label{prop:preclassification}
		Let $f_n\in\Sch(\R)$ satisfy $f_n\to Q_\rho$ in $H^1(\R)$ as $n\to\infty$, let $z_n$ be the unique zero supplied by Lemma~\ref{lem:unique-zero}, and denote
		\begin{equation*}
			r_n:=\mathbf B_{z_n}(f_n).
		\end{equation*}
		Then $\{z_n\}$ is bounded and away from zero. Moreover, for every convergent subsequence $\{z_{n_k}\}\subset \{z_n\}$, we have
		\begin{equation}\label{eq:spectral-zero-preclassification}
			z_{n_k}\to-\widetilde\rho\text{ for some }\widetilde\rho>0, \text{ and } r_{n_k}\to0\text{ in }H^1 \text{ as }n_k\to \infty.
		\end{equation}
	\end{prop}
	
	\begin{proof}
		Set
		\begin{equation*}
			\mathcal F:=\{f_1,f_2,\ldots\}.
		\end{equation*}
		Since $f_n\to Q_\rho$ in $L^2(\R)$, the set $\mathcal F$ is relatively compact in $L^2(\R)$. Moreover,
		\begin{equation*}
			\sup_n\|f_n(\cdot+h)-f_n\|_{L^2}
			\leq |h|\sup_n\|\partial_xf_n\|_{L^2}\to0\quad(\text{as }h\to0).
		\end{equation*}
		Lemma~\ref{lem:perturbation-determinant}(iv) yields, for every sufficiently small $\varepsilon>0$, a number $N=N(\varepsilon,\mathcal F)>0$ such that
		\begin{equation}\label{eq:spectral-zero-location-alternative}
			|z|\leq\frac N{\varepsilon^2}\quad\text{or}\quad \arg z\leq C_{M(\mathcal F)}\varepsilon.
		\end{equation}
		Choose $\varepsilon$ sufficiently small that $C_{M(\mathcal F)}\varepsilon<\pi/2$. Since $z_n\in\Sigma_{\rm II}$, the second alternative in \eqref{eq:spectral-zero-location-alternative} is impossible. Hence
		\begin{equation*}
			|z_n|\leq C.
		\end{equation*}
		
		Let $z_{n_k}\to z_*$ as $n_k\to\infty$ along an arbitrary convergent subsequence. We show that $z_*=-\wt{\rho}$ for some $\wt{\rho}>0$. If $z_*=0$, by \eqref{eq:backlund-density-bound-identities}, for every fixed $R>0$,
		\begin{equation}\label{eq:local-backlund-source-bound}
			\|r_{n_k}-G_{n_k}^2f_{n_k}\|_{L^2(-R,R)}=\|S_{n_k}\|_{L^2(-R,R)}\leq4(2R)^{1/2}|z_{n_k}|^{1/2}\to 0.
		\end{equation}
		Since $M(Q_{\rho})=4\pi$, we can choose $R$ so large that
		\begin{equation*}
			\int_{-R}^R|Q_\rho|^2\,dx>3\pi.
		\end{equation*}
		Then \eqref{eq:local-backlund-source-bound}, $|G_{n_k}|=1$, and $f_{n_k}\to Q_\rho$ in $L^2$ as $n_k\to\infty$ imply
		\begin{equation*}
			\liminf_{n_k\to\infty}M(r_{n_k})>3\pi.
		\end{equation*}
		On the other hand, by \eqref{eq:backlund-mass-shift} and $\arg z_{n_k}>\pi/2$, we obtain
		\begin{equation*}
			\limsup_{n_k\to\infty}M(r_{n_k})\leq4\pi-4\frac\pi2=2\pi.
		\end{equation*}
		This is a contradiction. Hence $z_* \neq 0$ and 
		\begin{equation}\label{eq:spectral-zero-two-sided-bound}
			0<c_0\leq|z_n|\leq C.
		\end{equation}
		
		If $z_*\in\C^+$, by Lemma \ref{lem:determinant-continuity} and \eqref{eq:algebraic-profile-trivial-determinant},
		\begin{equation*}
			0=\lim_{n_k\to\infty}a(z_{n_k};f_{n_k})=a(z_*;Q_\rho)=1,
		\end{equation*}
		which is impossible. Therefore, we obtain $z_*=-\widetilde\rho$ for some $\widetilde\rho>0$.
		
		Next, we prove that $r_{n_k}\to0$ in $H^1(\R)$ as $n_k\to\infty$. It follows from \eqref{eq:backlund-mass-shift} that
		\begin{equation}\label{eq:removed-field-mass-vanishing}
			M(r_{n_k})=M(f_{n_k})-4\arg z_{n_k}\to 0 \quad\text{as } n_k\to \infty.
		\end{equation}
		A direct calculation shows that
		\begin{equation}\label{eq:algebraic-profile-hamiltonian-vanishing}
			\mathcal H_2(Q_\rho)=0,
		\end{equation}
		where $\mathcal{H}_2$ is defined in \eqref{eq:second-hamiltonian}. Since $f_{n_k}\to Q_\rho$ in $H^1$ as $n_k\to\infty$,
		\begin{equation*}
			\mathcal H_2(f_{n_k})\to\mathcal H_2(Q_\rho)=0 \quad \text{as }n_k\to \infty.
		\end{equation*}
		By \eqref{eq:backlund-hamiltonian-shift},
		\begin{equation}\label{eq:removed-field-hamiltonian-vanishing}
			\mathcal H_2(r_{n_k})=\mathcal H_2(f_{n_k})-8\im(z_{n_k}^2)\to0 \quad \text{as }n_k\to\infty.
		\end{equation}
		By H\"older's and Gagliardo--Nirenberg's inequalities,
		\begin{equation*}
			\left|\frac{3i}{4}\int_\R|r_{n_k}|^2(r_{n_k}\overline{\partial_x r_{n_k}}-\overline{ r_{n_k}} \partial_x r_{n_k})\,dx\right|\leq\frac32\|r_{n_k}\|_6^3\|\partial_x r_{n_k}\|_{L^2}\leq CM(r_{n_k})\|\partial_x r_{n_k}\|_{L^2}^2.
		\end{equation*}
		Thus, 
		\begin{equation*}
			\bigl|\mathcal H_2(r_{n_k})-\|\partial_xr_{n_k}\|_{L^2}^2\bigr|
			\leq C\bigl(M(r_{n_k})+M(r_{n_k})^2\bigr)\|\partial_xr_{n_k}\|_{L^2}^2.
		\end{equation*}
		For sufficiently large $n_k$, by \eqref{eq:removed-field-mass-vanishing},
		\begin{equation}\label{eq:small-mass-hamiltonian-coercivity}
			\frac12\|\partial_x r_{n_k}\|_{L^2}^2 \leq \mathcal H_2(r_{n_k})\leq\frac32\|\partial_x r_{n_k}\|_{L^2}^2.
		\end{equation}
		It follows from \eqref{eq:removed-field-mass-vanishing}, \eqref{eq:removed-field-hamiltonian-vanishing} and \eqref{eq:small-mass-hamiltonian-coercivity} that $r_{n_k}\to 0 $ in $H^1$ as $n_k\to\infty$.
	\end{proof}
	
	We have not shown that $\wt{\rho}=\rho$. To prove this identity, we use the following converse to Proposition~\ref{prop:preclassification}.

	\subsection{Endpoint inverse B\"acklund compactness}

	\begin{prop}\label{prop:endpoint-inverse}
		Let $f_n\in \mathcal{S}(\R)$, and let $z_n\in \C^+$ be a zero of $a(z;f_n)$. Assume that $z_n\to-\wt{\rho}$ for some $\wt{\rho}>0$ as $n\to\infty$. If
		\begin{equation}\label{eq:removed-field-energy-smallness}
			r_n:=\mathbf B_{z_n}(f_n)\to0\quad\text{in }H^1(\R) \text{ as }n\to \infty,
		\end{equation}
		then
		\begin{equation}\label{eq:endpoint-inverse-orbit-convergence}
			\operatorname{dist}_{H^1}(f_n,\mathcal O_{\wt\rho})\to0\quad \text{as }n\to\infty.
		\end{equation}
	\end{prop}
	
	The proof is divided into Lemmas~\ref{lem:endpoint-phase-reduction}--\ref{lem:endpoint-global-reconstruction}.
	
	\begin{lem}\label{lem:endpoint-phase-reduction}
		Under the assumptions of Proposition~\ref{prop:endpoint-inverse}, let $G_n,S_n$ be the B\"acklund data associated with $z_n,f_n$ given in Definition \ref{defn:backlund-transform}, and write $z_n=|z_n|e^{i\varphi_n}$, $0<\varphi_n<\pi$. Then there exists a unique continuous function $\theta_n:\R\to(-\pi,\pi)$ such that
		\begin{equation*}
			G_n=e^{i\theta_n},\quad
			\theta_n(-\infty)=-\varphi_n,\quad
			\theta_n(+\infty)=\varphi_n.
		\end{equation*}
		Define
		\begin{equation*}
			p_n:=-\bar G_nS_n.
		\end{equation*}
		Then
		\begin{equation}\label{eq:phase-reduction-remainder-smallness}
			f_n-p_n\to0\quad\text{in }H^1(\R) \text{ as }n\to \infty.
		\end{equation}
	\end{lem}
	
	\begin{proof}
		Let $\mu_n=\sqrt{z_n}$ and let $d_n$ be defined in Definition \ref{defn:backlund-transform}. Since $d_n\neq0$ and $G_n=\bar d_n/d_n\in\mathbb S^1\setminus\{-1\}$, there exists a unique continuous lift $\theta_n$ such that
		\begin{equation}\label{eq:G-n-theta-n}
			G_n=e^{i\theta_n},\quad \theta_n(-\infty)=-\varphi_n,\quad \theta_n(+\infty)=\varphi_n,\quad -\pi<\theta_n<\pi.
		\end{equation}
		We also have $|\theta_n|\leq \varphi_n$. Indeed, by \eqref{eq:backlund-algebraic-differential-identities},
		\begin{equation}\label{eq:backlund-source-amplitude-identity}
			0\leq |S_n|^2 = 8|z_n|(\cos\theta_n-\cos\varphi_n),
		\end{equation}
		which implies that $|\theta_n|\leq \varphi_n$ and $\|S_n\|_\infty\leq4\sqrt{|z_n|}$.
		
		By \eqref{eq:backlund-transform-data},
		\begin{equation}\label{eq:inverse-backlund-decomposition}
			f_n-p_n=\bar G_n^2r_n.
		\end{equation}
		Since $r_n\to 0$ in $H^1(\R)$ as $n\to\infty$ and $|G_n|=1$,
		\begin{equation}\label{pn-qn-2}
			\|f_n-p_n\|_{L^2}\to 0 \quad\text{as }n\to\infty.
		\end{equation}
		By \eqref{eq:G-n-theta-n}, $\partial_x\big(\bar G_n^2r_n\big)=-2i\theta_n'\bar{G}_n^2r_n + \bar{G}_n^2r_n'$. By \eqref{eq:backlund-transform-data}, we also have
		\begin{equation}\label{eq:backlund-mass-density-decomposition}
			|f_n|^2-|r_n|^2=|S_n|^2-2\re(\bar G_nr_n\bar S_n).
		\end{equation}
		It follows from \eqref{eq:backlund-density-bound-identities} and \eqref{eq:backlund-mass-density-decomposition} that
		\begin{equation}\label{eq:backlund-phase-equation}
			\theta_n'=\frac12|S_n|^2-\re(\bar G_nr_n\bar S_n)=4|z_n|(\cos\theta_n-\cos\varphi_n)+e_n,
		\end{equation}
		where
		\begin{equation*}
			e_n=-\re(\bar G_nr_n\bar S_n),\quad |e_n|\leq|S_n||r_n|.
		\end{equation*}
		Since $z_n\to -\wt{\rho}$ as $n\to\infty$, we have $\varphi_n\to \pi$ as $n\to\infty$. By \eqref{eq:backlund-mass-shift},
		\begin{equation*}
			\|f_n\|_{L^2}^2=\|r_n\|_{L^2}^2+4\varphi_n\to4\pi \quad\text{as }n\to\infty.
		\end{equation*}
		Then,
		\begin{equation}\label{eq:Sn-L2}
			\|S_n\|^2_{L^2}=\|\bar G_nr_n-G_nf_n\|^2_{L^2}\to 4\pi \quad\text{as }n\to\infty.
		\end{equation}
		Hence
		\begin{equation*}
			\|e_n\|_1\leq\|S_n\|_{L^2}\|r_n\|_{L^2}\to0\quad\text{as }n\to\infty.
		\end{equation*}
		Similarly,
		\begin{equation*}
			\|e_n\|_{L^2}\leq\|S_n\|_\infty\|r_n\|_{L^2}\to0\quad\text{as }n\to\infty,
		\end{equation*}
		and
		\begin{equation}\label{eq:backlund-phase-derivative-bound}
			\|G_n'\|_{L^2}=\|\theta_n'\|_{L^2}
			\leq\frac12\||S_n|^2\|_{L^2}+\|e_n\|_{L^2}\leq\frac12\|S_n\|_\infty\|S_n\|_{L^2}+o(1)\leq C_{\wt{\rho}}.
		\end{equation}
		By Sobolev's inequality, $\|r_n\|_{L^{\infty}}\to 0$ as $n\to\infty$, and then
		\begin{equation}\label{eq:inverse-profile-remainder-smallness}
			\|\partial_x(f_n-p_n)\|_{L^2}=\|\partial_x\big(\bar{G}_n^2r_n\big)\|_{L^2}\leq 2\|\theta_n'\|_{L^2}\|r_n\|_\infty + \|r_n'\|_{L^2} \to0,
		\end{equation}
		as $n\to\infty$. Therefore, \eqref{eq:phase-reduction-remainder-smallness} follows from \eqref{pn-qn-2} and \eqref{eq:inverse-profile-remainder-smallness}.
	\end{proof}
	
	\begin{lem}\label{lem:endpoint-centered-local-limit}
		Under the assumptions and notation of Lemma~\ref{lem:endpoint-phase-reduction}, there exist translations $x_n\in\R$ and phases $\alpha_n\in\R$ such that, after replacing
		\begin{equation*}
			(f_n,r_n,S_n)(x)
			\quad\text{by}\quad
			e^{i\alpha_n}(f_n,r_n,S_n)(x+x_n),
		\end{equation*}
		we have $\theta_n(0)=0$, $S_n(0)>0$, and, as $n\to\infty$,
		\begin{equation}\label{eq:centered-local-limits}
			\theta_n\to\theta_*,\
			G_n\to G_*\text{ locally uniformly},
			\quad S_n\to S_*\text{ in }H^1_{\rm loc},
		\end{equation}
		where
		\begin{equation*}
			\theta_*(x)=2\arctan(4\wt\rho x),\quad
			G_*(x)=\frac{1+4i\wt\rho x}{1-4i\wt\rho x},\quad
			S_*(x)=\frac{4\sqrt{\wt\rho}\,e^{2i\wt\rho x}}{1+4i\wt\rho x}.
		\end{equation*}
	\end{lem}
	
	\begin{proof}
		Since $\theta_n(-\infty)=-\varphi_n<0<\varphi_n=\theta_n(+\infty)$, there exists $x_n\in\R$ such that $\theta_n(x_n)=0$. By \eqref{eq:backlund-source-amplitude-identity},
		\begin{equation*}
			|S_n(x_n)|^2=8|z_n|(1-\cos\varphi_n)>0.
		\end{equation*}
		Choose $\alpha_n=-\arg S_n(x_n)$. After the translation by $x_n$ and the simultaneous phase rotation by $e^{i\alpha_n}$, we retain the same notation. Then $\theta_n(0)=0$, $S_n(0)>0$, and
		\begin{equation}\label{eq:centered-source-amplitude-limit}
			S_n(0)=\{8|z_n|(1-\cos\varphi_n)\}^{1/2}\to4\sqrt{\wt\rho}\quad\text{as }n\to\infty.
		\end{equation}
		These transformations preserve the B\"acklund identities.
		
		Let $\theta_*$ be the solution of
		\begin{equation*}
			\theta_*'=4\wt\rho(1+\cos\theta_*),\quad \theta_*(0)=0.
		\end{equation*}
		Then
		\begin{equation}
			\theta_*(x)=2\arctan(4\wt\rho x).
		\end{equation}
		Subtracting the equations for $\theta_n$ and $\theta_*$ and integrating from $0$ to $x$, we obtain
		\begin{equation}\label{theta-n-theta*}
			\begin{aligned}
				\theta_n(x)-\theta_*(x)
				&=\int_0^x\Big[
				4|z_n|\bigl(\cos\theta_n(y)-\cos\theta_*(y)\bigr)\\
				&\qquad+4\bigl(|z_n|-\wt\rho\bigr)\bigl(1+\cos\theta_*(y)\bigr)\\
				&\qquad-4|z_n|\bigl(1+\cos\varphi_n\bigr)+e_n(y)
				\Big]\,dy .
			\end{aligned}
		\end{equation}
		Since $|\cos u-\cos v|\leq|u-v|$, Gronwall's inequality on $[0,R]$ and $[-R,0]$ yields
		\begin{equation*}
			\sup_{|x|\leq R}\bigl|\theta_n(x)-\theta_*(x)\bigr|
			\leq C_{R,\wt{\rho}}\bigl[\|e_n\|_{L^1(-R,R)}+\bigl||z_n|-\wt\rho\bigr|+|\cos\varphi_n+1|\bigr]\to0
		\end{equation*}
		as $n\to\infty$. Hence, $\theta_n\to\theta_*$ locally uniformly as $n\to\infty$. Since $G_n=e^{i\theta_n}$,
		\begin{equation}\label{eq:gauge-factor-local-limit}
			G_n\to G_*(x):=\frac{1+4i\wt\rho x}{1-4i\wt\rho x}\quad\text{locally uniformly as }n\to\infty.
		\end{equation}
		
		By \eqref{eq:backlund-algebraic-differential-identities} and \eqref{eq:inverse-backlund-decomposition},
		\begin{equation}\label{eq:backlund-source-spatial-ode}
			S_n'=-2i|z_n|\bar G_nS_n+\mathcal R_n,\quad \mathcal R_n:=-2i|z_n|(1-\bar G_n^2)r_n,
		\end{equation}
		and
		\begin{equation*}
			\|\mathcal R_n\|_{L^2}\leq4|z_n|\|r_n\|_{L^2}\to0 \quad\text{as }n\to\infty.
		\end{equation*}
		Let $S_*$ be the solution of 
		\begin{equation}\label{eq:S*equ}
			S_*'=-2i\wt\rho\bar G_*S_*,\quad S_*(0)=4\sqrt{\wt\rho}.
		\end{equation}
		Then
		\begin{equation}
			S_*(x)=\frac{4\sqrt{\wt\rho}\,e^{2i\wt\rho x}}{1+4i\wt\rho x}.
		\end{equation}	
		The same argument as in \eqref{theta-n-theta*} yields $S_n\to S_*$ locally uniformly as $n\to\infty$. Then \eqref{eq:backlund-source-spatial-ode} yields $S_n'\to S_*'$ in $L^2_{\rm loc}$ as $n\to\infty$. Hence $S_n\to S_*$ in $H^1_{\rm loc}$ as $n\to\infty$.
	\end{proof}
	
	The local convergence can be upgraded to global $H^1$-convergence by the mass identity and the first-order equation for $S_n$.
	
	\begin{lem}\label{lem:endpoint-global-reconstruction}
		Under the assumptions of Lemma~\ref{lem:endpoint-centered-local-limit},
		\begin{equation}\label{eq:global-source-and-profile-reconstruction}
			S_n\to S_*,\quad p_n\to-Q_{\wt\rho}
			\quad\text{in }H^1(\R)\ \text{as }n\to\infty.
		\end{equation}
	\end{lem}
	
	\begin{proof}
		A direct calculation shows that
		\begin{equation}\label{eq:limiting-source-mass}
			|S_*(x)|^2=\frac{16\wt\rho}{1+16\wt\rho^2x^2},\quad \|S_*\|_{L^2}^2=4\pi.
		\end{equation}
		By Lemma~\ref{lem:endpoint-centered-local-limit}, we have $S_n\rightharpoonup S_*$ in $L^2$ as $n\to\infty$. Thus, it follows from \eqref{eq:Sn-L2} that
		\begin{equation*}
			\|S_n-S_*\|^2_{L^2}=\|S_n\|_{L^2}^2+\|S_*\|_{L^2}^2- 2\re \langle S_n, S_* \rangle \to 0,
		\end{equation*}
		as $n\to\infty$. We obtain that
		\begin{equation}\label{eq:strong-source-mass-convergence}
			S_n\to S_*\quad\text{strongly in }L^2\text{ as }n\to\infty.
		\end{equation}
		For every $R>0$,
		\begin{equation*}
			\|\bar G_nS_n-\bar G_*S_*\|_{L^2}\leq \|S_n-S_*\|_{L^2}+\|(\bar G_n-\bar G_*)S_*\|_{L^2(-R,R)}+2\|S_*\|_{L^2(|x|>R)}.
		\end{equation*}
		Combining this estimate with \eqref{eq:gauge-factor-local-limit} and \eqref{eq:strong-source-mass-convergence}, and then letting $R\to\infty$, we obtain
		\begin{equation*}
			|z_n|\bar G_nS_n\to\wt\rho\bar G_*S_*\quad\text{in }L^2\text{ as }n\to\infty.
		\end{equation*}
		Hence, by \eqref{eq:backlund-source-spatial-ode} and \eqref{eq:S*equ},
		\begin{equation}\label{eq:source-derivative-convergence}
			S_n'\to S_*'\quad\text{in }L^2\text{ as }n\to\infty.
		\end{equation}
		Thus $S_n\to S_*$ in $H^1(\R)$ as $n\to\infty$. Set 
		\begin{equation*}
			p_*:=-\bar G_*S_*=-4\sqrt{\wt\rho}\,e^{2i\wt\rho x}\frac{1-4i\wt\rho x}{(1+4i\wt\rho x)^2}=-Q_{\wt\rho}.
		\end{equation*}
		Then
		\begin{equation*}
			p_n-p_*=-\bar G_n(S_n-S_*)-(\bar G_n-\bar G_*)S_*,
		\end{equation*}
		so \eqref{eq:gauge-factor-local-limit} and \eqref{eq:strong-source-mass-convergence} give $p_n\to p_*$ in $L^2$ as $n\to\infty$. 
		A direct calculation shows that
		\begin{equation}\label{eq:pn'}
			p_*'=-\bar{G}_*S_*'+i\bar{G}_* \theta_*' S_*, \quad p_n'=-\bar G_nS_n'+i\bar G_n\theta_n'S_n.
		\end{equation}
		Furthermore, by \eqref{eq:backlund-phase-equation},
		\begin{equation*}
			\theta_n'S_n=\frac12|S_n|^2S_n+e_nS_n,\quad \|e_nS_n\|_{L^2}\leq\|S_n\|_\infty^2\|r_n\|_{L^2}\to0\quad\text{as }n\to\infty.
		\end{equation*}
		The uniform $L^\infty$ bound of $S_n$ and \eqref{eq:strong-source-mass-convergence} imply
		\begin{equation*}
			\bigl\||S_n|^2S_n-|S_*|^2S_*\bigr\|_{L^2}\leq C_{\wt\rho}\|S_n-S_*\|_{L^2}\to0\quad\text{as }n\to\infty.
		\end{equation*}
		Since $\theta_*'=\frac12|S_*|^2$, it follows that $\theta_n'S_n\to\theta_*'S_*$ in $L^2$ as $n\to\infty$. Thus, by \eqref{eq:pn'}, we have $p_n'\to p_*'$ in $L^2$ as $n\to\infty$. Therefore,
		\begin{equation*}
			p_n\to p_* = -Q_{\wt{\rho}}\quad\text{in }H^1(\R)\text{ as }n\to\infty.
		\end{equation*}
	\end{proof}
	
	We are now in a position to prove Proposition \ref{prop:endpoint-inverse}.
	\begin{proof}[Proof of Proposition~\ref{prop:endpoint-inverse}]
		Apply the translations and phases from Lemma~\ref{lem:endpoint-centered-local-limit}. Since these symmetries preserve the distance to $\mathcal O_{\wt\rho}$, Lemmas~\ref{lem:endpoint-phase-reduction} and \ref{lem:endpoint-global-reconstruction} imply \eqref{eq:endpoint-inverse-orbit-convergence}.
	\end{proof}

	\begin{proof}[Proof of Proposition \ref{prop:scale}]
		Let $\{z_{n_k}\}$ be a convergent subsequence. It suffices to prove that $z_{n_k}\to-\rho$ as $n_k\to\infty$. By Proposition \ref{prop:preclassification},
		\begin{equation*}
			z_{n_k}\to -\wt\rho \quad \text{for some }\wt\rho>0\text{ as }n_k\to\infty.
		\end{equation*} 
		Applying Proposition~\ref{prop:endpoint-inverse} with $\widetilde\rho$ implies
		\begin{equation*}
			\operatorname{dist}_{H^1}(f_{n_k},\mathcal O_{\widetilde\rho})\to0\quad \text{as }n_k\to\infty.
		\end{equation*}
		Note that distinct scale orbits are positively separated. Indeed, suppose that, for some $\alpha\neq\beta$, there exist $\gamma_n,y_n\in\R$ such that
		\begin{equation}\label{eq:hypothetical-scale-orbit-coalescence}
			\|Q_\alpha-e^{i\gamma_n}Q_\beta(\cdot-y_n)\|_{H^1}\to0\quad\text{as }n\to\infty.
		\end{equation}
		The sequence $\{y_n\}$ is bounded; otherwise $Q_\beta(\cdot-y_n)\rightharpoonup0$ in $L^2$ along a subsequence, contradicting \eqref{eq:hypothetical-scale-orbit-coalescence}. After passing to a subsequence, $y_n\to y_*$ and $e^{i\gamma_n}\to e^{i\gamma_*}$. Then
		\begin{equation*}
			Q_\alpha=e^{i\gamma_*}Q_\beta(\cdot-y_*),
		\end{equation*}
		which implies $\alpha=\beta$, a contradiction. Therefore
		\begin{equation*}
			\operatorname{dist}_{H^1}(\mathcal O_\alpha,\mathcal O_\beta)>0\quad\text{for }\alpha\neq \beta.
		\end{equation*}
		Since $f_{n_k}\to Q_\rho$ in $H^1(\R)$ as $n_k\to\infty$, we obtain $\rho=\wt\rho$. Hence, \eqref{eq:endpoint-scale-selection} follows.
	\end{proof}
	
	\vspace{1cm}
	\section{Proof of Theorem \ref{thm:main}}\label{sec:proof-main-theorem}
	\vspace{0.5cm}
	
	In the notation introduced in Section~\ref{sec:spectral-analysis}, Theorem \ref{thm:main} is equivalent to the following theorem: 
	\begin{thm}\label{thm:full}
		Let $\rho>0$, and let $\{q_{0,n}\}_{n\geq 1} \subset H^1(\R)$ satisfy
		\begin{equation*}
			\operatorname{dist}_{H^1}(q_{0,n},\mathcal O_\rho)\to0 \quad\text{as }n\to\infty.
		\end{equation*}
		Let $q_n(t)$ be the solution of \eqref{eq:dnls-equation} with initial data $q_{0,n}$. Then
		\begin{equation}\label{eq:uniform-orbital-convergence}
			\sup_{t\in\R}\operatorname{dist}_{H^1}(q_n(t),\mathcal O_\rho)\to0\quad\text{as }n\to\infty.
		\end{equation}
	\end{thm}
	
	We first consider the Schwartz data and then extend to the entire energy space $H^1(\R)$.
	\begin{prop}[Schwartz stability]\label{prop:schwartz}
		Let $f_{0,n}\in\Sch(\R)$ and $f_{0,n}\to Q_\rho$ in $H^1(\R)$ as $n\to\infty$. Let $f_n(t)$ be the solution of \eqref{eq:dnls-equation} with initial data $f_{0,n}$. Then
		\begin{equation}\label{eq:uniform-schwartz-orbital-convergence}
			\sup_{t\in\R}\operatorname{dist}_{H^1}(f_n(t),\mathcal O_\rho)\to0\quad\text{as }n\to\infty.
		\end{equation}
	\end{prop}
	
	\begin{proof}
		We argue by contradiction. Suppose that there exist $\varepsilon_0>0$, a subsequence $\{n_k\}$, and times $t_{n_k}\in \R$ such that 
		\begin{equation}\label{eq:contradict}
			\operatorname{dist}_{H^1}(f_{n_k}(t_{n_k}),\mathcal O_\rho)\geq \varepsilon_0.
		\end{equation}
		
		Let $z_{n_k}$ be the unique simple zero of $a(z;f_{0,n_k})$ in $\Sigma_{\rm II}$ and set
		\begin{equation*}
			r_{0,n_k}:=\mathbf B_{z_{n_k}}(f_{0,n_k}).
		\end{equation*}
		It follows from Proposition~\ref{prop:scale} that
		\begin{equation*}
			z_{n_k}\to-\rho,\quad r_{0,n_k}\to0\text{ in }H^1(\R)\text{ as }n_k\to\infty.
		\end{equation*}
		As in \eqref{eq:small-mass-hamiltonian-coercivity}, for sufficiently large $n_k$, we have
		\begin{equation*}
			\frac12\|\partial_x r_{0,n_k}\|_{L^2}^2 \leq \mathcal H_2(r_{0,n_k})\leq\frac32\|\partial_x r_{0,n_k}\|_{L^2}^2.
		\end{equation*}
		Thus,
		\begin{equation*}
			\mathcal H_2(r_{0,n_k})\to 0 \quad\text{as }n_k\to\infty.
		\end{equation*}
		
		Let $r_{n_k}(t)$ be the solution of \eqref{eq:dnls-equation} with initial data $r_{0,n_k}$. By Lemma \ref{lem:darboux} and uniqueness,
		\begin{equation*}
			r_{n_k}(t)=\mathbf B_{z_{n_k}}(f_{n_k}(t)).
		\end{equation*}
		By mass conservation,
		\begin{equation*}
			M(r_{n_k}(t))=M(r_{0,n_k})\to 0\quad \text{as }n_k\to\infty.
		\end{equation*}
		As in \eqref{eq:small-mass-hamiltonian-coercivity}, mass conservation and $\mathcal H_2$ conservation imply that, for sufficiently large $n_k$,
		\begin{equation*}
			\sup_{t\in\R}\|\partial_x r_{n_k}(t)\|_{L^2}^2\leq 2\sup_{t\in\R}\mathcal H_2(r_{n_k}(t)) = 2\mathcal H_2(r_{0,n_k})\to0\quad\text{as }n_k\to\infty.
		\end{equation*}
		By Proposition~\ref{prop:endpoint-inverse}, 
		\begin{equation*}
			\operatorname{dist}_{H^1}(f_{n_k}(t_{n_k}),\mathcal O_\rho)\to0\quad\text{as }n_k\to\infty.
		\end{equation*}
		This contradicts \eqref{eq:contradict}.
	\end{proof}

	\begin{proof}[Proof of Theorem \ref{thm:full}]
		We argue by contradiction. Suppose that there exist $\varepsilon_0>0$, a subsequence $\{n_k\}$, and times $t_{n_k}\in \R$ such that 
		\begin{equation}\label{eq:thm-contradict}
			\operatorname{dist}_{H^1}(q_{n_k}(t_{n_k}),\mathcal O_\rho)\geq \varepsilon_0.
		\end{equation}
		
		After applying a phase rotation and a translation to $q_{0,n_k}$, we denote the resulting initial datum by $\wt q_{0,n_k}$ and may assume that $\wt q_{0,n_k}\to Q_\rho$ in $H^1(\R)$ as $n_k\to\infty$. Let $\wt q_{n_k}(t)$ be the corresponding solution. Then
		\begin{equation}\label{eq:thm-contradict-1}
			\operatorname{dist}_{H^1}(\wt q_{n_k}(t_{n_k}),\mathcal O_\rho)\geq \varepsilon_0.
		\end{equation}
		Set $T_{n_k}=|t_{n_k}|+1$. By Proposition \ref{prop:h1-well-posedness}, we can choose $f_{0,n_k}\in \mathcal{S}(\R)$ depending on $T_{n_k}$ such that 
		\begin{equation*}
			\|f_{0,n_k}-\wt q_{0,n_k}\|_{H^1}<\frac{1}{n_k},\quad \text{and} \quad \sup_{|t|\leq T_{n_k}}\|f_{n_k}(t)-\wt q_{n_k}(t)\|_{H^1}<\frac{1}{n_k},
		\end{equation*}
		where $f_{n_k}(t)$ is the solution of \eqref{eq:dnls-equation} with initial data $f_{0,n_k}$. Thus, it follows from Proposition \ref{prop:schwartz} that
		\begin{equation*}
			\operatorname{dist}_{H^1}(\wt q_{n_k}(t_{n_k}),\mathcal O_\rho) \leq \operatorname{dist}_{H^1}(f_{n_k}(t_{n_k}),\mathcal O_\rho) + \|f_{n_k}(t_{n_k})-\wt q_{n_k}(t_{n_k})\|_{H^1}\to 0
		\end{equation*}
		as $n_k\to\infty$, which contradicts \eqref{eq:thm-contradict-1}.
	\end{proof}

\end{document}